\documentclass[11pt]{amsart}

\usepackage{amssymb,latexsym}

\usepackage{enumerate}

\theoremstyle{plain} \numberwithin{equation}{section}
\newtheorem{Theorem}{Theorem}
\newtheorem{Lemma}[Theorem]{Lemma}
\newtheorem{Proposition}[Theorem]{Proposition}
\newtheorem{Corollary}[Theorem]{Corollary}
\newtheorem{Remark}[Theorem]{Remark}

\newtheorem{Criterion}[Theorem]{Criterion}

\newtheorem{Example}[Theorem]{Example}

\title[Criteria for existence of Riesz bases]
{Criteria for existence of Riesz bases consisting of root functions
of Hill and 1D Dirac operators}

\author[P. Djakov]{Plamen Djakov}

\author[B. Mityagin]{Boris Mityagin}
\thanks{B. Mityagin acknowledges the support of the Scientific and
Technological Research Council of Turkey and the hospitality of
Sabanci University, April--June, 2011.}

\begin{document}

\begin{abstract}
We study the system of root functions (SRF) of Hill operator $Ly =
-y^{\prime  \prime} +vy $ with a singular potential $v \in
H^{-1}_{per}$ and  SRF of 1D Dirac operator $  Ly = i
\begin{pmatrix} 1 & 0 \\ 0 & -1
\end{pmatrix}
\frac{dy}{dx}  + vy $ with matrix $L^2$-potential $v=\begin{pmatrix} 0 & P \\
Q & 0
\end{pmatrix},$
subject to periodic or anti-periodic boundary conditions. Series of
necessary and sufficient conditions (in terms of Fourier
coefficients of the potentials and related spectral gaps and
deviations) for SRF to contain a Riesz basis are proven.
Equiconvergence theorems are used to explain basis property of SRF
in $L^p$-spaces and other rearrangement invariant function spaces.
\vspace{1mm}\\
{\it Keywords}: Hill operators, singular potentials, Dirac
operators, spectral decompositions, Riesz bases, equiconvergence \vspace{1mm} \\
{\em 2010 Mathematics Subject Classification:} 47E05, 34L40.
\end{abstract}

\address{Sabanci University, Orhanli,
34956 Tuzla, Istanbul, Turkey}

 \email{djakov@sabanciuniv.edu}

\address{Department of Mathematics,
The Ohio State University,
 231 West 18th Ave,
Columbus, OH 43210, USA} \email{mityagin.1@osu.edu}

\maketitle

\section*{Content}
\begin{enumerate}

\item[Section 1.]  Introduction \vspace{3mm}

\item[Section 2.] Localization of spectra and Riesz
projections for Hill and Dirac operators   \vspace{3mm}

\item[Section 3.] Elementary geometry of bases in a Banach space
\vspace{3mm}

\item[Section 4.] Moving from geometric criterion to Hill and Dirac operators
\vspace{3mm}

\item[Section 5.] $L^p$-spaces and other rearrangement invariant function
spaces \vspace{3mm}

\item[Section 6.]  Constructive criteria in terms of Fourier
coefficients of potentials \vspace{3mm}

\item[Section 7.] Fundamental inequalities and criteria for Riesz
basis property

\end{enumerate}
\bigskip

\section{Introduction}

  1. In the case of ordinary differential operators with
  {\em strictly regular} boundary conditions ($bc$)
  on a finite interval the system
$\{u_k\}$ of eigen- and associated functions could contain only
finitely many linearly independent associated functions. The
well-defined decompositions
\begin{equation}
\label{1} \sum_k c_k (f) u_k = f \quad  \forall f \in L^2([0, \pi]),
\end{equation}
do converge; moreover, convergence is unconditional, i. e., $\{u_k
\}, \, \|u_k \|=1,$ is a Riesz basis in $L^2([0, \pi]).$
  These facts and phenomena have been well understood in the early 1960's
after the works of N. Dunford \cite{Du58,DS71}, V. P. Mikhailov
\cite{Mi62} and G. M. Keselman \cite{Ke64}.

Maybe the simplest case of {\em regular but not strictly regular
bc} comes if we consider a Hill operator $L_{bc}.$
\begin{equation}
\label{2}  Ly = - y^{\prime  \prime } + v(x) y, \quad 0 \leq x \leq
\pi,
\end{equation}
where $v(x) = v(x+\pi)$ is a complex-valued smooth function, and
$bc$ is periodic ($bc= Per+$) or anti-periodic  ($bc = Per-$), i.
e.,

(a) {\em periodic} $Per^+: \quad y(0) = y (\pi), \;\; y^\prime (0) =
y^\prime (\pi) $;

(b) {\em anti-periodic}  $Per^-: \quad  y(0) = - y (\pi), \;\;
y^\prime (0) = - y^\prime (\pi) $;

(Later we will consider non-smooth $v$ as well, say $v \in L^2$ or $
L^1,$ and $v \in H^{- 1/2}$ or $v \in H^{-1}_{per}$, -- see in
particular Section 4.1.)

  Recently, i.e., in the 2000's, many authors
  \cite{GT09, KeMa98,Ku06,Ma10,MaMe08,MaMe10,Ma06-1,Ma06-2,Ma06-3,Ma-ar,ShVe09} focused
on the problem of convergence of eigenfunction (or more generally
root function) decompositions in the case of regular but not
strictly regular $bc.$

 The free operators $L_{bc}^0=
d^2/dx^2,$ with $bc = Per^\pm$  have infinitely many double
eigenvalues $\lambda^0_n =n^2, \;$ (with $n$ even for $bc=Per^+$ and
$n$ odd for if $bc=Per^-$), the corresponding two-dimensional
eigenspaces $E_n^0 $ are mutually orthogonal and we have the
spectral decomposition of the space
 $$L^2 ([0,\pi]) = \oplus E_n^0    \quad \text{or} \quad f=
 \sum_n P_n^0 f \quad \forall f \in L^2 ([0,\pi]), $$
where $P_n^0$ is the orthogonal projection on $E_n^0.$ The operator
$L_{bc} (v) = L_{bc}^0 +v $  is a "perturbation" of the free
operator; its spectrum is discrete and for large enough $n, $ say
$n>N,$ close to $\lambda_n^0 =n^2$ there are exactly two eigenvalues
$\lambda_n^-, \lambda_n^+$ (counted with multiplicity). Moreover, if
$E_n$ is the corresponding two-dimensional invariant subspace and
$P_n = \frac{1}{2 \pi i} \int_{C_n} (z-\L_{bc})^{-1} dz $ is the
corresponding Cauchy projection, then we have the spectral
decomposition
\begin{equation}
\label{5} S_N f + \sum_{k>N} P_k f = f \quad  \forall f \in L^2 ([0,
\pi]),
\end{equation}
where $S_N $ is the (finite-dimensional) projection on the invariant
subspace corresponding to "small" eigenvalues of $L_{bc} (v),$ and
the series in (\ref{5}) converges unconditionally.

However, even if all eigenvalues $\lambda_n^-, \lambda_n^+, \; n>N $
are simple, there is a question  whether we could use the
corresponding eigenfunctions to give an expansion like (\ref{1}).
The same questions for $Per^\pm$ in the case of 1D periodic Dirac
operators could be asked. Interesting conditions on potentials $v$
(or on its Fourier coefficients), which guarantee basisness of
$\{u_k\}$, -- with or without additional assumptions about the
structure or smoothness of a potential $v$ -- have been given by A.
Makin \cite{Ma06-1,Ma06-2,Ma06-3,Ma-ar}, A. Shkalikov \cite{ShVe09},
O. Veliev \cite{Ve06,Ve10,VeDu02, DV05}, P. Djakov and B. Mityagin
\cite{DM15,DM20,DM25a,DM25,DM26}.
\bigskip

 2. In our papers \cite{KaMi01, DM5, DM15, DM21, DM26} we analyzed the relationship
between smoothness of a potential $v$ in (\ref{2}) and the rate of
"decay" of sequences of
\begin{equation}
\label{6} \text{{\em spectral gaps}}  \quad \gamma_n = \lambda^+_n
-\lambda^-_n
\end{equation}
and
\begin{equation}
\label{7} \text{{\em deviations}}  \quad \delta_n = \mu_n
-\frac{1}{2}( \lambda^+_n + \lambda^-_n).
\end{equation}
This analysis is based on the Lyapunov-Schmidt projection method:
by projecting on the $n$-th eigenvalue space $E_n^0$ of the free
operator $L^0 $ the eigenvalue equation $Ly =\lambda y $ is
reduced locally, for $\lambda = n^2 + z $ with $|z|<n/2 $ to an
eigenvalue equation for a $2 \times 2 $ matrix $\left [
\begin{array}{cc} \alpha_n (v,z)  & \beta^-_n (v;z)
\\ \beta^+_n (v;z) &  \alpha_n (v,z) \end{array}
\right ].$ The entries of this matrix are functionals (depending
analytically on $v$ and $z$),  which are given by explicit formulas
in terms of the Fourier coefficients of the potential $v$ (see
(\ref{22.34}) and (\ref{22.35}) below). They played a crucial role
in proving estimates for and inequalities between $\gamma_n, \,
\delta_n,\, \beta_n^\pm$ and
\begin{equation}
\label{9} t_n (z):=  |\beta^-_n (v,z)|/ |\beta^+_n (v,z)|
\end{equation}
-- see \cite{DM15}, Lemma 49 and Proposition 66.

Moreover, it turns out that there is an essential relation between
the Riesz basis property of the system of root functions and the
ratio functionals $t_n (v,z) $ which made possible to give
criteria for existence  of (Riesz) bases consisting of root
functions not only for Hill operators but for Dirac operators as
well (see, for example, \cite[Theorem 1]{DM25} or \cite[Theorem
2]{DM25a} for Hill, or \cite[Theorem 12]{DM26}  for Dirac
operators). These criteria are quite general and applicable to
wide classes of potentials. For example, we proved that if
\begin{equation}
\label{10} v(x) = 5e^{-4ix} +2 e^{2ix} -3e^{2ix} +4e^{4ix},
\end{equation}
then neither for $bc = Per^+$ nor for $bc =Per^-$ the root function
system of $L_{bc}$ contains a basis in $L^2 ([0, \pi])$. To apply
our criterion we had to overcome a few analytic difficulties. This
was done on the basis of our results and techniques from
\cite{DM10}.

In this paper we extend and slightly generalize these criteria. We
claim, both for Hill operators with singular
$H^{-1}_{per}$-potentials and Dirac operators with $L^2$-potentials
the following.

{\em Criterion. The root system of functions of the  operator
$L_{Per^\pm} (v) $ has the Riesz basis property (i.e., contains a
Riesz basis) if and only}
\begin{equation}
\label{20} \exists C>0: \quad  1/C \leq   t_n (z_n^*)  \leq C  \quad
\text{if}  \quad \lambda_n^- \neq   \lambda_n^-,  \quad n\in
\Gamma_{bc}, \;  |n| >N_*.
\end{equation}
(See the definition of $\Gamma_{bc}$ in Section 2, Formulas
(\ref{22.12}) and (\ref{22.112}).)
\bigskip

  3. Recently  F. Gesztesy and V. Tkachenko \cite[Theorem 1.2]{GT11}
 gave --  {\em in the case of Hill operators with
$L^2$ potentials} - a criterion of basisness in the following form:

  {\em The system of root vectors for $bc = Per^+$ or $bc = Per^-$, contains
a Riesz basis if and only if}
\begin{equation}
\label{11} R_{bc} = \sup \left \{\frac{|\mu_n -
\lambda_n^+|}{|\lambda_n^+ -\lambda_n^-|}:\; n \in \Gamma_{bc}, \;
\lambda_n^+ \neq \lambda_n^- \right \} <\infty.
\end{equation}
One can prove, by using the  estimates  of $|\lambda_n^+
-\lambda_n^-|$ and  $ |\mu_n - \lambda_n^+|$ in terms of $|\beta^-_n
(v,z)|$ and $ |\beta^+_n (v,z)|$ (see \cite[Theorem 66, Lemma
49]{DM15} and \cite[Theorem 37, Lemma 21]{DM21}) that the conditions
(\ref{20}) and (\ref{11}) are equivalent.

However, we directly show (see Theorem \ref{thmM} in Section 7),
using the fundamental inequalities proven in \cite{KaMi01, DM5,
DM15, DM21}, that (\ref{11}) gives necessary and sufficient
conditions of Riesz basisness of root system with $bc = Per^+$ or
$bc = Per^-$ both

(A) {\em in the case of 1D periodic Dirac operators with $L^2$
potential,}

and

(B) {\em in the case of  Hill operators with potential in}
$H^{-1}_{per}$.
\bigskip

4. Criterion for $L^p$-spaces,  $1 < p < \infty, $ given in
\cite[Theorem 1.4]{GT11} can be essentially improved and extended as
well. We take any separable rearrangement invariant function space
$E$ on $[0, \pi]$ (see \cite{KPS82,LT-2}) squeezed between $L^a$ and
$L^b, \, 1 < a \leq b < \infty.$ {\em If
\begin{equation}
\label{13}  1/a - 1/b < 1/2
\end{equation}
in the above cases (A) and (B) the root function system contains a
basis in $E$ if and only (\ref{11})  holds}. In the case of Hill
operators with $v \in H^{-1/2}$ the hypothesis (\ref{13}) could be
weakened to
\begin{equation}
\label{14}  1/a - 1/b < 1.
\end{equation}

Of course for $L^p, \, 1 < p < \infty$, we can put $a = b = p,$ so
(\ref{13}) and (\ref{14}) hold.

The structure of this paper and the topics discussed in different
sections are shown in Content, see p. 1.
\bigskip

\section{Localization of spectra and Riesz
projections for Hill and Dirac operators}

For basic facts of Spectral Theory of ordinary differential
operators we refer to the books \cite{LS91,Na69,Ma86}. But let us
introduce some notations and remind a few properties of Hill and
Dirac operators on a finite interval.
\bigskip

1. We consider the Hill operator
\begin{equation}
\label{22.1} Ly = -y^{\prime \prime} +v(x) y, \quad x \in I= [0,\pi],
\end{equation}
with a (complex-valued) potential $v \in L^2 (I),$ or more generally
with a singular potential $v \in H^{-1}_{per}$ of the form
\begin{equation}
\label{22.2} v  = w^\prime, \quad w \in L_{loc}^2 (\mathbb{R}),
\;\;w(x+\pi) = w(x).
\end{equation}
For $v \in L^2, $ we consider the following  {\em bc} (boundary
conditions):

(a) {\em periodic} $Per^+: \quad y(0) = y (\pi), \;\; y^\prime (0) =
y^\prime (\pi) $;

(b) {\em anti--periodic}  $Per^-: \quad  y(0) = - y (\pi), \;\;
y^\prime (0) = - y^\prime (\pi) $;

(c) {\em Dirichlet} $ Dir: \quad y(0) = 0, \;\; y(\pi) = 0. $

For each $bc=Per^\pm, \, Dir $ the operator $L$ generates a closed
operator $L_{bc}$ with
\begin{equation}
\label{22.3} Dom(L_{bc}) = \{f \in W_2^2 (I)  \; : \; \; f  \;\;
\text{satisfies} \;\; bc\}.
\end{equation}

In the case of singular potentials (\ref{22.2})  A. Savchuk and A.
Shkalikov \cite{SS00,SS03} suggested to use the quasi-derivative
$$
y^{[1]} = y^\prime  - w \,y
$$
in order to define properly the boundary conditions and
corresponding operators. In particular, the periodic and
anti--periodic boundary conditions $Per^\pm $ have the form

($a^*$) $\quad   Per^+: \quad y(\pi)= y(0), \quad y^{[1]}(\pi)=
y^{[1]}(0), $

($b^*$)  $\quad  Per^-: \quad y(\pi)= -y(0), \quad y^{[1]} (\pi)= -
y^{[1]} (0). $

The Dirichlet boundary condition has the same form (c) as in the
classical case. Of course, in the case where $w$ is a continuous
function, $Per^+ $ and $Per^-$  coincide, respectively, with the
classical periodic boundary condition ($a$) and ($b$).

We refer the reader to our papers \cite{DM17, DM16, DM21} for
definitions of the operators $L_{bc}$ and their domains in the case
of $H^{-1}_{per}$-potentials. (We followed \cite{SS00,SS03} and
further development of A. Savchuk -- A. Shkalikov's approach by R.
Hryniv and Ya. Mykytyuk \cite{HM01,HM04,HM06} to justify Fourier
method in analysis of Hill-Schr\"odinger operators with singular
potentials.

If $v=0$ we denote by $L^0_{bc}$ the corresponding free operator. Of
course, it is easy to describe the spectra and eigenfunctions for
$L^0_{bc}.$ Namely, we have

(i) $ Sp (L^0_{Per^+}) = \{n^2, \; n = 0,2,4, \ldots \};$ its
eigenspaces are  $E^0_n = Span \{e^{\pm inx} \} $ for $n>0 $ and
$E^0_0 = \{ const\}, \; \; \dim E^0_n = 2 $ for $n>0, $ and $\dim
E^0_0 = 1. $

(ii) $ Sp (L^0_{Per^-}) = \{n^2, \; n = 1,3,5, \ldots \};$ its
eigenspaces are  $E^0_n = Span \{e^{\pm inx} \}, $ and $ \dim E^0_n =
2. $

(iii) $ Sp (L^0_{Dir}) = \{n^2, \; n \in \mathbb{N} \};$ each
eigenvalue $n^2 $ is simple; the corresponding normalized
eigenfunction is
\begin{equation}
\label{22.4} s_n (x) = \sqrt{2} \sin nx,
\end{equation}
so the corresponding eigenspace is
\begin{equation}
\label{22.5} G_n^0 = Span \{ s_n \}.
\end{equation}
\bigskip

2. Localization of spectra in the case of Hill operators.

\begin{Proposition}
\label{prop2.1} (localization of spectra) Consider $L_{bc} (v)$ with
$bc=Per^\pm, \, Dir$ and with potential $v\in L^2 $ or $v\in
(\ref{22.2}).$ Then, for large enough $N_* = N_* (v)\in 2\mathbb{N},
$ we have
\begin{equation}
\label{22.8} Sp \, (L_{bc}) \subset \Pi_{N_*} \cup \bigcup_{n>N_*, \,
n\in \Gamma_{bc}} D(n^2, r_n),
\end{equation}
where
\begin{equation}
\label{22.9}\Pi_{N}= \{z=x+iy\in \mathbb{C}: \; |x|,\,|y| < N^2 +
\frac{1}{2}N,
\end{equation}
\begin{equation}
\label{22.10} D(a,r) = \{z\in \mathbb{C}: \; |z-a|< r\},
\end{equation}
with
\begin{equation}
\label{22.11} r_n = N_*/2 \quad \text{if} \;\; v \in L^2, \qquad
r_n=n/4 \quad \text{if} \;\; v \in H^{-1}_{per},
\end{equation}
and
\begin{equation}
\label{22.12} \Gamma_{bc}= \begin{cases} \{ 0\}\cup 2\mathbb{N} & bc=
Per^+,\\  2\mathbb{N}-1 & bc= Per^-,\\
\mathbb{N} & bc=Dir.
\end{cases}
\end{equation}
With the resolvent $R(z) = (z-L_{bc})^{-1}$ well defined in the
complement of $Sp \, (L_{bc}),$ we set
\begin{equation}
\label{22.13} S_{N_*} = \frac{1}{2\pi i}\int_{\partial \Pi_{N_*}}
(z-L_{bc})^{-1} dz,
\end{equation}
\begin{equation}
\label{22.14} P_n = \frac{1}{2\pi i}\int_{|z-n^2|=r_n}
(z-L_{bc})^{-1} dz, \quad n>N_*, \; n \in \Gamma_{bc},
\end{equation}
and
\begin{equation}
\label{22.15} S_N = S_{N_*} + \sum_{\tiny \begin{array}{c} n= N_*+1\\
n\in \Gamma_{bc}
\end{array}
}^N P_n.
\end{equation}
Then
\begin{equation}
\label{22.16} \dim P_n = \begin{cases}  2 & n \;\text{even}, \; bc=
Per^+, \\
2 & n \;\text{odd}, \; bc= Per^-, \\
1 & n \in \mathbb{N}, \; bc = Dir,
\end{cases}
\end{equation}
and
\begin{equation}
\label{22.17} \dim S_{N_*} = \begin{cases} N_* +1 & bc = Per^+,\\
N_* &   bc = Per^- \;\; \text{or} \; Dir.
\end{cases}
\end{equation}
In each case the series
\begin{equation}
\label{22.18} S_{N_*} f +\sum_{n>N_*, \, n \in \Gamma_{bc}} P_n f = f
\quad \forall f \in L^2 (I)
\end{equation}
converges unconditionally, so the system of projections is a Riesz
system.
\end{Proposition}
The latter is true not only for potentials $v \in L^2$ but in the
case $v\in H^{-1}_{per} $ as well. It has been proven by A. Savchuk
and A. Shkalikov \cite[Theorem 2.8]{SS03}. An alternative proof is
given by the authors in \cite{DM16}, see Theorem 1 and Proposition
8.
\bigskip

3. Next we remind  the basic fact about spectra decompositions and
spectral decompositions for Dirac operators
\begin{equation}
\label{12.1} Ly = i
\begin{pmatrix}
1 & 0 \\ 0 & -1
\end{pmatrix}
\frac{dy}{dx}  + vy,
\end{equation}
\begin{equation}
\label{12.2} v(x)=
\begin{pmatrix}
0  & P(x) \\ Q(x) & 0
\end{pmatrix},
\quad y =
\begin{pmatrix}
y_1\\ y_2
\end{pmatrix},
\end{equation}
with $L^2$--potential $v,$ i.e., $P,Q \in L^2 (I). $

We consider three types of boundary conditions:

(a) {\em periodic} $Per^+ : \quad y(0) = y(\pi),$ i.e., $ y_1 (0) =
y_1 (\pi) $ and $ y_2 (0) = y_2 (\pi); $

(b) {\em anti-periodic} $Per^- : \quad y(0) = - y(\pi),$ i.e., $ y_1
(0) = -y_1 (\pi) $ and $ y_2 (0) = -y_2 (\pi); $

(c) {\em Dirichlet} $Dir: \quad y_1 (0) = y_2 (0), \; y_1 (\pi) = y_2
(\pi).$

The corresponding closed operator with a domain
\begin{equation}
\label{11.3a}
 \Delta_{bc} = \left \{f \in (W_1^2 (I))^2  \; : \;
\; F =
\begin{pmatrix} f_1 \\ f_2 \end{pmatrix}
 \in (bc)  \right \}
\end{equation}
 will be denoted by $L_{bc}.$
 If $v=0,$ i.e., $P\equiv 0, Q \equiv 0,$ we write
$L^0_{bc}.$  Of course, it is easy to describe the spectra and
eigenfunctions for $L^0_{bc}:$

(a)  $Sp (L^0_{Per^+}) = \{n \;\text{even} \} = 2 \mathbb{Z}; $ each
number $n \in 2\mathbb{Z} $ is a double eigenvalue, and the
corresponding eigenspace is
\begin{equation}
\label{12.5} E^0_n = Span \{e^1_n, e^2_n\},
\end{equation}
 where
\begin{equation}
\label{12.6}
 \displaystyle e^1_n (x) =
\begin{pmatrix}
e^{-inx}\\ 0
\end{pmatrix},
\quad e^2_n (x) =
\begin{pmatrix}
0\\ e^{inx}
\end{pmatrix};
\end{equation}

(b)  $Sp (L^0_{Per^-}) = \{n \;\text{odd} \} = 2 \mathbb{Z}+ 1; $ the
corresponding eigenspaces $E_n^0 $ are given by (\ref{12.5}) and
(\ref{12.6}) but $ n \in 2\mathbb{Z} +1; $

(c) $ Sp (L^0_{Dir}) = \{n \in \mathbb{Z} \};$ each eigenvalue $n $
is simple. The corresponding normalized eigenfunction is
\begin{equation}
\label{12.7} g_n  (x) = \frac{1}{\sqrt{2}} \left (e^1_n + e^2_n
\right ), \quad n \in \mathbb{Z},
\end{equation}
so the corresponding (one-dimensional) eigenspace is
\begin{equation}
\label{12.8} G_n^0 = Span \{ g_n \}.
\end{equation}
\bigskip

4. Localization of spectra in the case of Dirac operators.

\begin{Proposition}
\label{prop2.2} (localization of spectra)  For Dirac operators
$L_{bc} (v)$ with $bc=Per^\pm, \, Dir,$  there is $N_* = N_* (v), $
such that
\begin{equation}
\label{22.108} Sp \, (L_{bc}) \subset \Pi_{N_*} \cup \bigcup_{n>N_*,
\, n\in \Gamma_{bc}} D(n^2, 1/4),
\end{equation}
where
\begin{equation}
\label{22.109}\Pi_{N}= \{z=x+iy\in \mathbb{C}: \; |x|,\,|y| < N^2 +
\frac{1}{4},
\end{equation}
and
\begin{equation}
\label{22.112} \Gamma_{bc}= \begin{cases} 2\mathbb{Z} & bc=
Per^+,\\  1+2\mathbb{Z} & bc= Per^-,\\
\mathbb{Z} & bc=Dir.
\end{cases}
\end{equation}
With the resolvent $R(z) = (z-L_{bc})^{-1}$ well defined in the
complement of $Sp \, (L_{bc}),$ we set
\begin{equation}
\label{22.113} S_{N_*} = \frac{1}{2\pi i}\int_{\partial \Pi_{N_*}}
(z-L_{bc})^{-1} dz,
\end{equation}
\begin{equation}
\label{22.114} P_n = \frac{1}{2\pi i}\int_{|z-n|=1/4} (z-L_{bc})^{-1}
dz, \quad |n|>N_*, \; n \in \Gamma_{bc},
\end{equation}
and
\begin{equation}
\label{22.115} S_N = S_{N_*} + \sum_{\tiny \begin{array}{c}  N_*+1 \leq |n| \leq N\\
n\in \Gamma_{bc}
\end{array}
} P_n.
\end{equation}
Then
\begin{equation}
\label{22.116} \dim P_n = \begin{cases}  2 & n \;\text{even}, \; bc=
Per^+, \\
2 & n \;\text{odd}, \; bc= Per^-, \\
1 & n \in \mathbb{Z}, \; bc = Dir,
\end{cases}
\end{equation}
and
\begin{equation}
\label{22.117} \dim S_{N_*} = \begin{cases} 2N_* +2 & bc = Per^+,\\
2N_* &   bc = Per^- \\ 2N_* +1  & bc= Dir.
\end{cases}
\end{equation}
In each case the series
\begin{equation}
\label{22.118} S_{N_*} f +\sum_{|n|>N_*, \, n \in \Gamma_{bc}} P_n f
= f \quad \forall f \in L^2 (I)
\end{equation}
converges unconditionally, so
\begin{equation}
\label{22.119} \{S_{N_*}, \;\;   P_n, \; |n|> N_*, \; n\in
\Gamma_{bc} \}
\end{equation}
is a Riesz system of projections.
\end{Proposition}

The latter is proven in \cite[Theorem 5.1]{DM20}. (Under more
restrictive assumption on the potential  $v \in H^\alpha, \; \alpha>
1/2, $ the fact that (\ref{22.119}) is a Riesz system of projections
has been proven in \cite[Theorem 8.8]{Mi04}.)

Propositions \ref{prop2.1} and \ref{prop2.2} guarantee the existence
of the level $N_*=N_* (v) $ when all formulas for $P_n, \, S_N, $
etc. become valid if $n > N_*, \, n \in \mathbb{N}$ (or $|n| > N_*,
\, n \in \mathbb{Z} $ in the Dirac case). In the next sections,
there are other formulas which are valid for large enough $n$ and
require different levels $N_*=N_* (v). $ But throughout the paper we
use one and the same letter $N_* $ to indicate by the inequalities
$n>N_*$ or $|n|>N_*$ that formulas hold for sufficiently large
indices.
\bigskip

5. Propositions  \ref{prop2.1} and \ref{prop2.2} allows us to apply
the Lyapunov--Schmidt projection method (see \cite[Lemma 21]{DM15})
and reduce the eigenvalue equation $Ly = \lambda y $ to a series of
eigenvalue equations in two-dimensional eigenspaces $E_n^0$ of the
free operator.

This leads to the following (see for Hill operators \cite[Section
2.2]{DM15} in the case $L^2$-potentials, and \cite[Lemma 6]{DM21} in
the case of $H^{-1}_{per}$-potentials; for Dirac operators, see
\cite[Section 2.4]{DM15}).

\begin{Lemma}
\label{lem1}

 (a) Let $L$ be a Hill operator with a potential
$v\in L^2$ or $v \in H^{-1}_{per}.$ Then, for large enough $n\in
\mathbb{N},$ there are functions $\alpha_n (v,z) $ and $ \beta^\pm_n
(v;z), \; |z| < n $ such that a number $\lambda = n^2 + z, \;|z| <
n/4, $ is a periodic (for even $n$) or anti-periodic (for odd $n$)
eigenvalue of  $L$ if and only if $z$ is an eigenvalue of the matrix
\begin{equation}
\label{p1}  \left [
\begin{array}{cc} \alpha_n (v,z)  & \beta^-_n (v;z)
\\ \beta^+_n (v;z) &  \alpha_n (v,z) \end{array}
\right ].
\end{equation}

 (b) Let $L$ be a Dirac operator with a potential
$v\in L^2.$ Then, for large enough $|n|, \; n\in \mathbb{Z},$ there
are functions $\alpha_n (v,z) $ and $ \beta^\pm_n (v;z), \; |z| < 1
$ such that a number $\lambda = n + z, \;|z| < 1/4, $ is a periodic
(for even $n$) or anti-periodic (for odd $n$) eigenvalue of  $L$ if
and only if $z$ is an eigenvalue of the matrix (\ref{p1}).

(c) A number $\lambda = n^2 + z^*, \;|z| < n/4, $ (respectively,
$\lambda = n + z, \;|z| < 1/4 $ in the Dirac case) is a periodic
(for even $n$) or anti-periodic (for odd $n$) eigenvalue of $L$ of
geometric multiplicity 2  if and only if $z^*$  is an eigenvalue of
the matrix (\ref{p1})  of geometric multiplicity 2.
\end{Lemma}

The functionals $\alpha_n (z;v) $ and $\beta^\pm_n (z;v)$ are well
defined for large enough $|n|$ by explicit expressions in terms of
the Fourier coefficients of the potential (see for Hill operators
with $L^2$-potentials \cite[Formulas (2.16)-(2.33)]{DM15}, for Dirac
operators \cite[Formulas (2.59)--(2.80)]{DM15}, and for Hill
operators with $H^{-1}_{per}$-potentials \cite[Formulas
(3.21)--(3.33)]{DM21}).

Here we provide formulas only for $\beta^\pm_n (v;z)$ in the case of
Hill operators with $H^{-1}_{per}$-potentials. Let $v $ be a
singular potential as in (\ref{22.2}), and
\begin{equation}
\label{22.54} v= w^\prime, \;  \quad  w= \sum_{m \in 2\mathbb{Z}}
W(m) e^{imx}.
\end{equation}
Then the Fourier coefficients of $v$ are given by
\begin{equation}
\label{22.55} V(m) = i m \, W(m), \quad m \in 2\mathbb{Z},
\end{equation}
and by \cite[Formulas (3.21)--(3.33)]{DM21} we have
\begin{equation}
\label{22.34} \beta^\pm_n (v;z) = V(\pm 2n) + \sum_{k=1}^\infty
S^\pm_k (n,z),
\end{equation}
with
\begin{equation}
\label{22.35} S^\pm_k (n,z) = \sum_{j_1, \ldots, j_k \neq \pm n}
\frac{V(\pm n-j_1)V(j_1 - j_2) \cdots V(j_{k-1} -j_k) V(j_k  \pm n)}
{(n^2 -j_1^2 +z) \cdots (n^2 - j_k^2 +z)}.
\end{equation}

Next we summarize some basic properties of $\alpha_n (z;v) $ and
$\beta^\pm_n (z;v).$

\begin{Proposition}
\label{bprop1} Let $v $ be a $ H^{-1}_{per} $--potential of the form
(\ref{22.2}), and let $L_{Per^\pm} $ be the corresponding Hill
operator.

(a) The functionals $\alpha_n (z;v) $ and $\beta^\pm_n (z;v) $
depend analytically on $z$  for $|z|< n. $ There exists a sequence
of positive numbers  $\varepsilon_n \to 0 $ such that for large
enough $n$
\begin{equation}
\label{1.136} |\alpha_n (v;z)|+|\beta^\pm_n (v;z)|  \leq
 n \cdot \varepsilon_n,  \quad  \; |z| \leq n/2,
\end{equation}
and
\begin{equation}
\label{1.137} \left |\frac{\partial\alpha_n}{\partial z} (v;z)
\right| + \left |\frac{\partial\beta^\pm_n}{\partial z} (v;z) \right
|\leq \varepsilon_n, \quad   \; |z| \leq n/4.
\end{equation}

(b) For large enough $n$ (even, if $bc= Per^+$ or odd, if $bc=
Per^-$), a number $\lambda = n^2+ z, \, $ $|z|< n/4, $ is an
eigenvalue of $L_{Per^\pm} $ if and only if $z$ satisfies the basic
equation
\begin{equation}
\label{be1} (z-\alpha_n (z;v))^2 = \beta^+_n (z;v) \beta^-_n (z;v).
\end{equation}

(c) For large enough $n,$ the equation (\ref{be1}) has exactly two
roots in the disc $|z| < n/4 $ counted with multiplicity.

\end{Proposition}

\begin{proof}
Part (a) is proved in \cite[Proposition 15]{DM21}. Lemma~\ref{lem1}
implies Part (b).  By (\ref{1.136}), $\sup \{|\frac{1}{z}\alpha_n
(z)|, \; |z|= n/4 \} \to 0 $ and $\sup \{|\frac{1}{z}\beta^\pm_n
(z)|,\;|z|= n/4 \} \to 0. $ Therefore, Part (c) follows from the
Rouch\'e theorem.
\end{proof}

\begin{Proposition}
\label{bprop} Let $L_{Per^\pm} $ be a Dirac operator with
$L^2$-potential.

(a) The functionals $\alpha_n (z;v) $ and $\beta^\pm_n (z;v) $
depend analytically on $z$  for $|z| < 1. $ There exists a sequence
of positive numbers  $\varepsilon_n \to 0 $ such that for large
enough $|n|$
\begin{equation}
\label{1.36} |\alpha_n (v;z)|+ \, |\beta^\pm_n (v;z)| \leq
\varepsilon_n, \quad  |z| \leq 1/2,
\end{equation}
and
\begin{equation}
\label{1.37} \left |\frac{\partial\alpha_n}{\partial z} (v;z) \right|
+ \left |\frac{\partial\beta^\pm_n}{\partial z} (v;z) \right |\leq
\varepsilon_n, \quad |z| \leq 1/4.
\end{equation}

(b) For large enough $|n|,$ ($n$ even, if $bc= Per^+$ or odd, if $bc=
Per^-$), the number $\lambda = n+ z, $ $ z\in D=\{\zeta: |\zeta| \leq
1/4\}, $ is an eigenvalue of $L_{Per^\pm} $ if and only if
 $z\in D $ satisfies the basic equation
\begin{equation}
\label{be} (z-\alpha_n (z;v))^2 = \beta^+_n (z;v) \beta^-_n (z,v),
\end{equation}

(c) For large enough $|n|,$ the equation (\ref{be}) has exactly two
(counted with multiplicity) roots in $D.$
\end{Proposition}

\begin{proof}
Part (a) is proved in \cite[Proposition 35]{DM15}. Lemma~\ref{lem1}
implies Part (b).  By (\ref{1.36}), $\sup_D |\alpha_n (z)| \to 0 $
and $\sup_D |\beta^\pm_n (z)| \to 0 $ as $n\to \infty. $ Therefore,
Part (c) follows from the Rouch\'e theorem.
\end{proof}
\bigskip

\section{Elementary geometry of bases in a Banach space}

In this section we give a few well-known facts about geometry and
bases in Banach and Hilbert spaces -- see
\cite{KaSa89,LT-1,LT-2,Di92, KPS82}.

1. Let  $\{u_k \in X, \; \psi_k \in X^\prime\}_{k\in \mathbb{N}}$
be a biorthogonal system in a Banach space $X$, i. e.,
\begin{equation}
\label{3.1} \psi_k (u_j) =\begin{cases} 1, &  k=j,\\
0,  &  k\neq j
\end{cases} \quad   j,k \in \mathbb{N}.
\end{equation}
The system $\{u_k\}$ is called {\em a basis, or a Shauder basis in}
$Y$, its closed linear span if
\begin{equation}
\label{3.2} \lim_{N\to \infty} \sum_{k=1}^N \psi_k (y) u_k =y, \quad
\forall y \in Y.
\end{equation}
Put
\begin{equation}
\label{3.3} Q_m = q_{2m-1} +q_{2m},  \quad \text{where} \quad q_j
(x) = \psi_j (x) u_j, \;\; j \in \mathbb{N}
\end{equation}
are one-dimensional projections so
\begin{equation}
\label{3.4} \|q_j\| = \|u_j\| \cdot \|\psi_j\|.
\end{equation}
Let us assume that
\begin{equation}
\label{3.5} \lim_{M\to \infty}  \sum_{m=1}^M  Q_m y = y \quad
\forall y \in Y.
\end{equation}
In this case, certainly
\begin{equation}
\label{3.5a} \sup_m \|Q_m\| = C < \infty.
\end{equation}

Notice that partial sums in (\ref{3.5}) are equal to partial sums in
(\ref{3.2}) with even indices. But
\begin{equation}
\label{3.6} \sum_{k=1}^{2t+1} \psi_k (y) u_k = \left (\sum_{m=1}^t
Q_m y \right ) + \psi_{2t+1}(y) u_{2t+1}.
\end{equation}
These elementary identities together with (\ref{3.1}) explain the
following.
\begin{Lemma}
\label{lem3.1} If $\{u_k\}_1^\infty$ is a basis in $Y,$ i.e.,
(\ref{3.2}) holds then
\begin{equation}
\label{3.7} T \equiv \sup_j \|q_j\| < \infty.
\end{equation}
Under the assumption  (\ref{3.5}) if  (\ref{3.7}) holds then $\{u_k
\}_1^\infty$ is a basis in $Y.$
\end{Lemma}
\bigskip

2. What does happen inside of 2D subspaces $E_m =Ran \, Q_m, \; m
\in \mathbb{N}?$

Let $\{u_1, u_2\},\; \|u_j\|=1, $ be a basis in $E_m$ and  let
$\psi_1, \psi_2 $ be the corresponding biorthogonal functionals, so
\begin{equation}
\label{3.14} h= \psi_1 (h) u_1 + \psi_2 (h) u_2 \quad \forall h \in
E_m.
\end{equation}
 To avoid any confusion
let us notice that for $j=2m-1, 2m $
\begin{equation}
\label{3.16} \psi_j (y)  =\psi_j (Q_m y) \quad \forall y \in Y,
\end{equation}
and if (\ref{3.5}) holds then with (\ref{3.5a})
\begin{equation}
\label{3.17} \|Q_m y \| \leq C \|y\|.
\end{equation}
Therefore,
\begin{equation}
\label{3.18} \|\psi_j\| \geq \sup\{|\psi_j (w)|: \; \|w\|=1, \; w
\in E_m\}
\end{equation}
$$
\geq \sup\{|\langle \frac{1}{C} Q_m y, \psi_j  \rangle|: \; \|y\|=1,
\; y \in Y\} =\frac{1}{C} \|\psi_j\|,
$$
so
\begin{equation}
\label{3.19} \|\psi_j\| \leq C \kappa_j, \quad \kappa_j \leq
\|\psi_j\|,
\end{equation}
i.e.,
\begin{equation}
\label{3.19a} \kappa_j \equiv \|\psi_j|E_m\| \leq \|\psi_j|Y\| \leq
C \kappa_j.
\end{equation}
In a Hilbert space case, elementary straightforward estimates show
that for $j = 1, 2$
\begin{equation}
\label{3.15} \kappa_j = \sup\{|\psi_j (w)| : \; \|w\|=1, \; w \in
\mathbb{C}^2\} = \left (1- |\langle u_1, u_2 \rangle|
\right)^{-1/2}.
\end{equation}
We use this fact when analyzing subspaces $E_m$ and their bases
$\{u_{2m-1}, u_{2m}\},$ $m \in \mathbb{N}.$
\bigskip

3. Now we consider separable Hilbert spaces $H$. We say that the
system $\{Q_m\} \in (\ref{3.3})$ is {\em a Riesz system, or an
unconditional 2D-block basis in} $Y$ if for some $C > 0$
\begin{equation}
\label{3.8} \|\sum_{m \in F} Q_m\|\leq C  \quad \text{for any finite
subset}  \;\; F \subset \mathbb{N}.
\end{equation}
\begin{Lemma}
\label{lem3.2} Assume the system of 2D projections ${Q_m} \in
(\ref{3.3})$ in a Hilbert space $H$ is a Riesz system, i. e.,
(\ref{3.8}) holds. If $\{u_k\}_1^\infty $ is a basis in $Y \subset
H$ then it is an unconditional basis in $Y.$
\end{Lemma}

\begin{proof}
Proof is based on the Orlicz \cite{Or33} lemma:
\begin{Lemma}
\label{lem3.3} (\ref{3.8}) holds for the system ${Q_m} \in
(\ref{3.3})$ in a Hilbert space if and only if for some constant
$C_1
> 0$
\begin{equation}
\label{3.9}  \frac{1}{C_1^2} \|y\|^2 \leq \sum_m \|Q_m y\|^2 \leq
C_1^2 \|y\|^2 \quad \forall y\in Y.
\end{equation}
\end{Lemma}

By Lemma \ref{lem3.1} and (\ref{3.7}), (\ref{3.4}) the norms of 1D
projections $q_j$ are uniformly bounded. By (\ref{3.1})
\begin{equation}
\label{3.10} q_j Q_m = \begin{cases} q_j    &     \text{if} \;\;j=
2m-1, \,2m\\
0  &  \text{otherwise}
\end{cases}
\end{equation}
so for $j= 2m-1, \, 2m $
\begin{equation}
\label{3.11} \|q_j y\| = \| q_j Q_m y\| \leq M \| Q_m y\| \leq
(\|q_{2m-1}y\| +\|q_{2m}y\|) \leq 2 M \| Q_m y\|.
\end{equation}
Therefore
\begin{equation}
\label{3.13} \frac{1}{4M^2} \left ( \|q_{2m-1}y\|^2 +\|q_{2m}y\|^2
\right )  \leq \|Q_m y\|^2  \leq 2M^2 \left ( \|q_{2m-1}y\|^2
+\|q_{2m}y\|^2 \right )
\end{equation}
and with $C_1 = 2M$ the condition (\ref{3.8}) holds for the system
of 1D projections $\{q_j\}.$  It guarantees that $\{q_j\}$ is a
Riesz system and $\{u_k\}$ is an unconditional basis in $Y$.
\end{proof}
\bigskip

4. Now we are ready to claim the following.

\begin{Criterion}
\label{crit6} With notations (\ref{3.1}), (\ref{3.3}) let us assume
that the system of 2D projections $\{Q_m\}$ is a Riesz system in a
Hilbert space. If a normalized system
\begin{equation}
\label{3.20} \{u_k\}, \quad \|u_k\|=1,
\end{equation}
is a basis in $Y$ then
\begin{equation}
\label{3.21}  \kappa := \sup \,\{(1-|\langle u_{2m-1}, u_{2m}
\rangle|^2 )^{-1/2}: \; m \in \mathbb{N}\} <\infty.
\end{equation}
If the condition (\ref{3.21}) holds then $\{u_k\}$ is a normalized
unconditional basis, that is a Riesz basis in $Y$.
\end{Criterion}

\begin{Corollary}
\label{cor7} If (\ref{3.8}) holds in a Hilbert space $H$ the system
$\{u_k\}_1^\infty \in (\ref{3.20}), (\ref{3.1})$ is a Riesz basis if
and only if it is a basis.
\end{Corollary}

\section{Moving from geometric criterion to Hill and Dirac operators}

1. The basic assumption in the geometric Criterion \ref{crit6} is
{\em the property of a system of projections $\{Q_m\}$ in a Hilbert
space to be a Riesz system.}

When we analyze systems of projections $\{P_n, \; |n|\geq N_*\}$
coming from Hill or Dirac operators, then it is a fundamental fact
that {\em they are Riesz systems}.

If $v \in L^2$ this has been understood since 1980's
(\cite{Sh79,Sh82,Sh83}). To make technically formal reference let us
mention \cite[Proposition 5]{DM5}, where it is shown that
\begin{equation}
 \label{2.1}
\|P_n -P_n^0\|_{2 \to \infty}  \leq C \frac{\|v\|_2}{n},
\end{equation}
so certainly
\begin{equation}
 \label{2.2}
\sum_{|n|>N} \|P_n -P_n^0\|^2_{2 \to 2}  < \infty
\end{equation}
and with
\begin{equation}
 \label{2.3}
\dim S_{N} =\dim S^0_{N}
\end{equation}
the Bari-Markus theorem \cite[Ch.6, Sect. 5.3, Theorem 5.2]{GoKr69}
implies that the series converge unconditionally.

A. Savchuk and  Shkalikov \cite[Theorem 2.4]{SS03} showed that
(\ref{2.2}) - (\ref{2.3}) hold if $v \in H^{-1}_{per} $ and $bc=
Per^\pm.$ An alternative proof has been given by the authors -- see
Theorem 1 and Proposition 8 in \cite{DM19}.

Finally, in the case of one dimensional Dirac operators we proved
(\ref{2.2}) - (\ref{2.3}) if $v \in L^2 $ and $bc = Per^\pm$ or
$Dir$ (see \cite{DM20}, Theorems 3.1 and 5.1). Later  we proved
(\ref{2.2}) - (\ref{2.3}) for {\em arbitrary} regular boundary
condition  -- see Theorems 15 and 20 in \cite{DM23}; however,
 we do not use these results from \cite{DM23} in the present paper.
Certainly in all these cases
\begin{equation}
 \label{2.4}
\|P_n -P_n^0\|_2  \to 0 \quad \text{and} \quad \|P_n\|_2 \leq 3/2
\quad \text{for} \;\; |n|>N_*.
\end{equation}

These bibliography references justify applicability of Criterion
\ref{crit6} when we are trying to give different analytic criteria
for Riesz basis property of the root function system of specific
differential operators.

Of course, Corollary \ref{cor7} indicates that in a Hilbert space
there is no separate question about Schauder basis property. If
$\{Q_m\},$ or $\{S_N; \; P_n, \, |n|\geq N \}$ is a Riesz system
such that $\dim Q_m =2, \; \dim P_n =2, $  then the properties of
the system $\{u_{2m-1}, \, u_{2m}\}$ to be a Riesz basis or to be a
Schauder basis are identical. Therefore, to talk about two
properties is semantically artificial.
\bigskip

 2. Let us define the root function system $\{u_j \}$
  which will play a special role in our analysis in
Sections 5 and 6 and in Main Theorem (Theorem \ref{thmM}). Section 3
and Criterion \ref{crit6} use an indexation by natural numbers, i.
e., $m \in \mathbb{N}$. But in the case of {\em Riesz bases} (or
{\em unconditional convergence of series}) it means that we can
ignore the ordering in $\mathbb{N}$, consider any countable set of
indices and use all related statements from Section 3. Of course, in
the case of bases which are not Riesz bases we should be accurate
when we use statements from Section 3  -- this is important in
Section 6.

\begin{Remark}
\label{rem13} In the case of Hill operators, $\Gamma_{bc}  \in
(\ref{22.12})$ as a subset of $\mathbb{N}$ has a natural ordering
and we have no confusion in defining the sum in (\ref{22.18}) --
this is
$$ \lim_{N \to \infty }
\sum_{\tiny \begin{array}{c}  N_* < n \leq N\\
n\in \Gamma_{bc} \end{array}}  $$ if this limit does exists. However
for Dirac operators $\Gamma_{bc} \in (\ref{22.112}) $ are subsets in
$\mathbb{Z};$ we have to accept convention to define the sum in
(\ref{22.118}) as
$$
\lim_{N \to \infty } \sum_{\tiny \begin{array}{c}  N_* < |n| \leq N \\
 n\in \Gamma_{bc} \end{array}} \qquad \text{and} \quad
\lim_{N \to \infty } \sum_{\tiny \begin{array}{c}  -N < n \leq N+1
\\ n\in \Gamma_{bc}, \,|n|>N_* \end{array}}
$$
if both these limits exist and are equal. Such understanding is in
accordance with the choice of contours in (\ref{22.109})  and
(\ref{22.113}).

\end{Remark}

 But in all {\em four cases} --
 $Per^+$ and $Per^-$ for both Hill and Dirac operators -- {\em the
 systems of projections
\begin{equation}
\label{42.1} \{S_{N_*}, \;   P_n, \,|n| >N_*, \, n \in \Gamma_{bc}
\}
\end{equation}
given in (\ref{22.13}) - (\ref{22.17}) or (\ref{22.113}) -
(\ref{22.117}) are Riesz systems of projections} as (\ref{22.18})
and (\ref{22.118}) tell us.

Now we define three sets of indices:
\begin{equation}
\label{42.2} \mathcal{M} =\{m \in \Gamma_{bc}: \; |m|>N_*, \;
\lambda_m^+ - \lambda_m^- \neq 0 \},
\end{equation}
\begin{equation}
\label{42.3} \mathcal{M}_1 =\{m \in \Gamma_{bc}: \; |m|>N_*, \;
\lambda_m^+ - \lambda_m^- = 0, \;\; P_m L_{bc} P_m = \lambda \cdot
1_{E_m} \},
\end{equation}
i. e., ${\, \lambda_m^+ \;}$ is a double eigenvalue of algebraic and
geometric multiplicities 2;
\begin{equation}
\label{42.4} \mathcal{M}_2 =\{m \in \Gamma_{bc}: \; |m|>N_*, \;
\lambda_m^+ - \lambda_m^- = 0, \;\; P_m L_{bc} P_m \quad \text{is a
Jordan matrix} \},
\end{equation}
i. e., $\, \lambda_m^+ \; $ is a double eigenvalue of algebraic
multiplicity 2 and geometric multiplicity 1.

If $m \in \mathcal{M}, $ we choose $(u_{2m-1}, \, u_{2m})$ in such a
way that
\begin{equation}
\label{42.6} Lu_{2m} =\lambda_m^+ u_{2m}, \quad Lu_{2m-1}
=\lambda_m^- u_{2m-1},
\end{equation}
\begin{equation}
\label{42.6a}  \|u_j \| =1,  \quad j \in \mathbb{N}.
\end{equation}

If $m \in \mathcal{M}_1$ choose any pair of orthogonal normalized
vectors in $E_m$
\begin{equation}
\label{42.7} \langle  u_{2m-1}, \, u_{2m} \rangle =0.
\end{equation}
\bigskip

3. For $ m \in \mathcal{M}_2$  we consider two different options to
choose root functions for a basis. \vspace{3mm}

 {\bf Option 1.} If  $ m \in \mathcal{M}_2, $ then there is only one
 (up to constant factor) normalized eigenvector $f \in E_m,$
\begin{equation}
\label{42.8} L f = \lambda^+_m f, \quad \|f\| =1,
\end{equation}
so we choose
\begin{equation}
\label{42.9} u_{2m}=f, \quad  u_{2m-1} \perp u_{2m}, \quad
\|u_{2m-1}\|  =1.
\end{equation}
Such a pair $(u_{2m-1},\, u_{2m}), \; m \in \mathcal{M}_2$ -- as for
$ m \in \mathcal{M}_1 $ -- is a nice basis in $ E_m,$ so it will not
be an obstacle for Riesz basisness of the larger system (see Lemmas
\ref{lem3.2} and \ref{lem3.3}) which contains $\{u_{2m-1},
u_{2m}\}.$ \vspace{3mm}

{\bf Option 2.} We choose $u_{2m}$ as in Option 1, and we choose
$u_{2m-1} \in (\ref{42.7}) $ to be an {\em associated} function,
i.e.,
\begin{equation}
\label{42.10} L_{bc} u_{2m} = \lambda_m^+ u_{2m}, \quad L_{bc}
u_{2m-1} = \lambda_m^+ u_{2m-1} +u_{2m}.
\end{equation}
Since we choose $u_{2m-1}$ to satisfy (\ref{42.10}) and
(\ref{42.7}), it is uniquely defined but its norm $\|u_{2m-1}\|$ is
out of our control.

For Hill operators with potentials in $L^1$ A. Shkalikov and O.
Veliev \cite[Theorem 1, Step 1]{ShVe09} observed that if $M_2$ is
infinite then
\begin{equation}
\label{42.11} \|u_{2m-1}\| \to \infty \quad  \text{as} \;\;  m \to
\infty, \; m \in \mathcal{M}_2.
\end{equation}
For potentials $v \in L^2$ this has been proven in \cite[Ine.
(3.29)]{KaMi01}.  Formula (\ref{42.11}) implies that $\{u_{2m-1}, \,
u_{2m}, \; m \in \mathcal{M}_2\}$ could not be a subset of a Riesz
basis.

However, if a potential $v $ is singular it may happen that
$\mathcal{M}_2$ is infinite but with the choices determined by
Option 2 we have
\begin{equation}
\label{42.12}  \exists C>0 \quad  0 < \frac{1}{C} \leq  \|u_{2m-1}\|
\leq C <\infty, \; \forall m \in \mathcal{M}_2.
\end{equation}

\begin{Example}
\label{eG} Take Gasymov type \cite{Ga80} singular potential
\begin{equation}
\label{42.13} v(x) = \sum_{k=1}^\infty  c(k) e^{2 i k x },
\end{equation}
with
\begin{equation}
\label{42.14} \exists A > 0: \quad
 1/A \leq |c(k)| \leq A  \quad \forall k \in \mathbb{N}.
\end{equation}
Then we have:

(i) $\mathcal{M}_2 = \Gamma_{bc}\cap \{n: n>N_*\} $  for  $bc =
Per^+$ and $Per^-,$ i. e., all  $E_m $ with $m>N_*$ are Jordan;

(ii)  with choices by Option 2 the condition (\ref{42.12}) holds,
and the system of eigen- and associated functions $\{u_{2m-1} ,
u_{2m}\}$ is a Riesz basis in $L^2$.

\end{Example}

This example is in a quite curious contrast with the case $v \in
L^2$ or $ v \in L^1$ -- see (\ref{42.11}) above. We prove the claims
(i) and (ii) in Section 6, where other examples of $H^{-1}_{per}$
potentials are considered as well.
\bigskip

4. Now we declare our canonical choice of vectors in Jordan blocks:

\begin{equation}
\label{500}  from \; now \; on \; \;our \; special \; system \;
\{u_j\} \;is\; chosen\; by \; Option \;1 \; .
\end{equation}

\begin{Remark}
\label{BR} The choice (\ref{500}) guarantees that the total system
$\{u_j\}$  of root functions has the Riesz basis property if and only
if its subsystem
\begin{equation}
\label{501} U_\mathcal{M} =\{u_{2m-1} , u_{2m}, \; m \in
\mathcal{M}\}
\end{equation}
is a Riesz basis in its closed linear span.
\end{Remark}

But still we need to define $u_j$ for small $j, \; |j| \leq N_*. $
 This system will be a basis in $E_* = Ran S_{N_*}.$
  Of course $\dim E_* < \infty,$ so this choice has
no bearing on whether the entire system will or will not be a Riesz
basis (or a basis) in $L^2$ or another function space. We want it to
be a system of root functions, so we choose the system of eigen- and
associated functions of a finite-dimensional operator $ S_* L_{bc}
S_*, \; S_* = S_{N_*}$ (We omit elementary linear algebra details.)
\bigskip

\section{$L^p$-spaces and other rearrangement invariant function
spaces}

1. In Sections 3 and 4 we discussed (criteria of) convergence of
decompositions
\begin{equation}
\label{5.1} S_{N^*} f + \sum_{n>N^*, n \in \Gamma_{bc}} P_n f = f
\quad \forall f \in L^2
\end{equation}
in $L^2.$ Convergence of such series or of eigenfunction
decompositions in $L^p, \, p \neq 2,$ or other {\em rearrangement
invariant function spaces} (see \cite{KPS82, MS96}) is not an
independent from convergence in $L^2$ question because of the
following two reasons of very general nature:

(A) In the case of free operator $L^0$ its decompositions
(\ref{5.1}) are standard (or slight variations of) Fourier series.
{\em These decompositions
\begin{equation}
\label{5.2} S^0_{N^*} f + \sum_{n>N^*, n \in \Gamma_{bc}} P^0_n f =
f \quad \forall f \in E
\end{equation}
converge in $E$ if $E$ is a separable rearrangement invariant
function space where Hilbert transform is bounded}. This is
certainly the case if
\begin{equation}
\label{5.3}
  L^a \supset E \supset L^b \quad \text{for some}\;\;
  a, b \;\; \text{with} \;\; 1 < a \leq b < \infty.
\end{equation}
See \cite[Theorem 2.7.2]{KPS82}, \cite{MS96,Zy90}, and more about
Boyd indices in \cite{LT-2}, Theorem 2.c.16 and Proposition 2.b.3
there.

(B) Put
\begin{equation}
\label{5.4} S_N = S_{N^*}  + \sum_{n>N^*, n \in \Gamma_{bc}}^N P_n.
\end{equation}
There are different versions of {\em equiconvergence} -- see the
survey paper of A. Minkin \cite{Min99}. For example, J. Tamarkin
\cite{Ta17, Ta28} and M. Stone \cite{St26} proved the following.
\begin{Lemma}
\label{Lem5.3} If $v \in L^1$ then for any $f \in L^1$
\begin{equation}
\label{5.5} \|(S_N - S_N^0)f\|_\infty \to 0.
\end{equation}
\end{Lemma}

This lemma helps to cover the case of Hill operator with $v \in
L^1.$  For $v \in H^{-1}_{per}$  see Proposition \ref{Prop5.5}
below.

Equiconvergence in the case of Dirac operator with potentials $v \in
L^c, \; c >4/3, $  is proven in \cite[Theorem 6.2 (a)]{Mi04}. As a
corollary it is noticed there \cite[Theorem 6.4, (6.105)]{Mi04} that
the series (\ref{5.6}) converges in $L^p (I, \mathbb{C}^2) \; 1 < p
<\infty. $
\bigskip

2. Now we can combine (A) and (B) to conclude the following.
\begin{Proposition}
\label{Prop5.3} If $v \in L^2$ and (\ref{5.3}) holds then
\begin{equation}
\label{5.6}
 S_{N^*} f + \sum_{n>N_*, n \in \Gamma_{bc}} P_n f = f
\quad \forall f \in E
\end{equation}
\end{Proposition}

\begin{proof}
Indeed
\begin{equation}
\label{5.7} S_N f = S_N^0 f +(S_N - S_N^0)f
\end{equation}
but with (\ref{5.3}) $\|g\|_E \leq \|g\|_\infty$ so for $f \in L^1$
\begin{equation}
\label{5.8} \|(S_N - S_N^0)f\|_E \leq \|(S_N - S_N^0)f\|_\infty \to
0.
\end{equation}
Now (\ref{5.2}) and (\ref{5.8}) together imply (\ref{5.6}).

\end{proof}
\bigskip

3. Of course in the case of Hill operators we want to cover
potentials $v \in H^{-1}_{per}$ as well. This is possible  because
the following {\em equiconvergence} statement is true.
\begin{Proposition}
\label{Prop5.5} Let $v \in H^{-1}_{per}$,  $W$ be coming from
(\ref{22.54}) and (\ref{22.55}), and
\begin{equation}
\label{5.9}
  1 < a \leq b < \infty \quad \text{with} \quad \delta = 1/2 - (1/a - 1/b) > 0.
\end{equation}
Then for any $N > N_* (v)$
\begin{equation}
\label{5.10} \|S_N - S_N^0: \; L^a \to L^b\| \leq C(\delta) \left [
N^{-\tau} + \mathcal{E}_N (W) \right ],
\end{equation}
where
\begin{equation}
\label{5.11}
  \tau =\begin{cases}
\delta &  \text{if} \;\; 1<a \leq 2 \leq b <\infty;\\
1-1/a &  \text{if} \;\; 1<a \leq  b \leq 2;\\
1/b  &  \text{if} \;\; 2\leq a  \leq b <\infty.
  \end{cases}
\end{equation}
and
\begin{equation}
\label{5.12} \mathcal{E}_N (W) = \left ( \sum_{|m|\geq N} |W(m)|^2
\right )^{1/2}.
\end{equation}

\end{Proposition}

Proof with all details is to be given in another paper we will
submit shortly.

\begin{Proposition}
\label{Prop5.6} If $v \in H^{-1}_{per}$ and $E$ is a s.r.i.f.s. such
that (\ref{5.3}) and (\ref{5.9}) hold then (\ref{5.6}) hold.
\end{Proposition}

\begin{proof}
Now with $\|g\|_a \leq \|g\|_E  \leq \|g\|_b$ (\ref{5.10}) and
(\ref{5.7}) imply
\begin{equation}
\label{5.13} \|(S_N - S_N^0)f\|_E \leq \|(S_N - S_N^0)f\|_{L^b} \leq
\|S_N - S_N^0: \; L^a \to L^b \| \cdot \|f\|_{L^b} \leq \varepsilon
(N) \|f\|_E,
\end{equation}
where $$ \varepsilon (N) = C(\delta) \left [ N^{-\delta} +
\mathcal{E}_N (w) \right ] \to 0, $$ so (\ref{5.6}) holds.
\end{proof}
\bigskip

4. Terms $P_m f$ in (\ref{5.6}) are vectors in two-dimensional
subspaces
\begin{equation}
\label{5.14} E_m = Lin\,Span \{u_{2m-1}, u_{2m}\},
\end{equation}
with $\{u_j\}$ defined in Section 4.2, (\ref{500}).

{\bf Fact (C)}. In these 2D subspaces $L^1$ norms and $L^\infty$
norms are {\em uniformly equivalent,} i.e., with $B = B(v) < \infty$
\begin{equation}
\label{5.15} \|F\|_\infty \leq B \|F\|_1 \quad \text{if} \quad F \in
E_m, \; m \geq N (v)
\end{equation}
This is proven in \cite[Theorem 8.4, p.185]{Mi04} for Dirac
operators with $V \in L^p, \, 1 < p,$ and in \cite[Theorem 51,
p.159]{DM21} for Hill operators with $v \in H^{-1}_{per}.$

Section 4.2 explains that with conditions (\ref{3.5}) and
(\ref{3.5a})
$$
\|\psi_j|E_m\| \leq \|\psi_j|E\| \leq C \|\psi_j|E_m\|.
$$
-- see (\ref{3.17}) - (\ref{3.19a}).  By Lemma \ref{lem3.1}, the
system $\{u_j\}$ is a basis in $Y \subset E$ if and only if
\begin{equation}
\label{5.16} \sup_j \|u_j\|_E \cdot \|\psi_j\|_E <\infty.
\end{equation}
But Fact (C) shows that (\ref{5.16}) holds -- or does not hold --
for all s.r.i.f.s. $E$ such that
\begin{equation}
\label{5.17} L^1([0,\pi]) \supset E \supset L^\infty ([0,\pi])
\end{equation}
simultaneously. Any condition which is good to guarantee basisness
in one $E$ is automatically good for all $E's$. Therefore, we can
immediately to claim the following.

\begin{Theorem}
\label{thm5.8} Let $E$ be a separable r.i.f.s. and
\begin{equation}
\label{5.18} L^a ([0,\pi]) \supset E \supset L^b ([0,\pi]), \quad 1<
a \leq b <\infty.
\end{equation}
The system $\{u_j\}$ defined in (\ref{500}) is a basis in $E$ (or
$E^2$) if and only if $\{u_j\}$ is a basis in $ L^2 ([0,\pi])$ ( or
$ (L^2 ([0,\pi]))^2.$
\end{Theorem}
\bigskip

\section{Criteria in terms of Fourier
coefficients of potentials}

1. Let $L = L_{Per^\pm} (v) $ be a Hill operator with
$H^{-1}_{per}$-potential, or Dirac operator with $L^2$-potential,
subject to periodic $Per^+$ or anti-periodic $Per^+$ boundary
conditions.

Recall that the eigenvalues $ \lambda_n^\pm, \, \mu_n $ and the
related functions $\beta^\pm_n (v,z)$ are well defined for large
enough $|n|.$   Let
\begin{equation}
\label{t1} t_n (z) = \begin{cases}  |\beta^-_n (z)/\beta^+_n (z)|  &
\text{if} \quad \beta^+_n (z) \neq 0,\\
\infty &  \text{if} \quad \beta^+_n (z) = 0, \;\beta^-_n (z) \neq
0,\\
1 &  \text{if} \quad \beta^+_n (z) = 0, \;\beta^-_n (z)=0.
\end{cases} \qquad |n| > N_*.
\end{equation}
Then the following criterion for existence of a Riesz basis
consisting of root functions of $L$ holds.

\begin{Theorem}
\label{crit} Let $\mathcal{M} =\{n: \; |n| \geq N_*, \; \lambda^-_n
\neq \lambda^+_n \},$ and  let $\{u_{2n-1}, \,u_{2n} \}$ be a pair
of normalized eigenfunctions corresponding to the eigenvalues
$\lambda_n^-, \, \lambda_n^+.$

(a) The system $\{u_{2n-1}, \,u_{2n},  \; n\in \mathcal{M}\} $ is a
Riesz basis in its closed linear span if and only if
\begin{equation}
\label{cr11} 0< \liminf_{n\in \mathcal{M}} t_n (z_n^*), \quad
\limsup_{n\in \mathcal{M}} t_n (z_n^*)  < \infty,
\end{equation}
where $z_n^* = \frac{1}{2} (\lambda^-_n + \lambda^+_n) -\lambda^0_n
$ with $\lambda^0_n = n^2 $ for Hill operators and $\lambda^0_n = n
$ for Dirac operators.

(b) The system of root functions of $L$ contains a Riesz basis if
and only if (\ref{cr11}) holds.
\end{Theorem}

This theorem implies that Condition (7) in Theorem \ref{thmM} is
equivalent to Conditions (1) - (6) there.

\begin{proof}
In view of  Remark \ref{BR} we need to prove only (a).

For Dirac operators, \cite[Theorem 12]{DM26} proves, in the case $
\mathbb{N} \setminus \mathcal{M} $ is finite, that Condition
(\ref{cr11}) implies the existence of a Riesz basis in $L^2
([0,\pi],\mathbb{C}^2)$ which consists of eigenfunctions and at most
finitely many associated functions of the operator $L_{Per^\pm}(v).$
The same proof explains that (\ref{cr11}) implies (a) for arbitrary
infinite set of indices $\mathcal{M}$ not only for Dirac operators
but also for Hill operators with $H^{-1}_{per}$-potentials.

If (\ref{cr11}) fails, then one may follow, with a slight
modification, the proof of \cite[Theorem 71]{DM15} in order to show
that (a) fails. We provide all details of such a modification below.

Suppose (\ref{cr11}) fails. Then there is a subsequence of indices
$(n_k) $ in $\mathcal {M}$ such that either
\begin{equation}
\label{7.21} t_{n_k} (z_{n_k}^*) \to 0 \quad \text{as} \;\; k \to 0,
\end{equation}
or $t_{n_k} (z_{n_k}^*) \to \infty.$ Next we consider only the case
(\ref{7.21}) because the other one is symmetric -- if $1/t_{n_k}
(z_n^*) \to 0, $ then one may exchange the roles of $\beta_n^+$ and
$\beta_n^-$ and use the same argument.

\begin{Lemma}
\label{lem7.1} In the above notations, if (\ref{7.21}) holds then
there is a sequence $(\eta_k)$ of positive numbers such that
\begin{equation}
\label{7.22} t_{n_k} (z) \leq  \eta_k \to 0 \quad  \forall \, z \in
[z_{n_k}^-,z_{n_k}^+],
\end{equation}
where $[z_n^-,z_n^+]$ denotes the segment with end points $z_n^- $
and $z_n^+.$
\end{Lemma}

\begin{proof}
By \cite[Lemma 20]{DM21} (in the case of Hill operators with
$H^{-1}_{per}$-potentials) or by \cite[Lemma 40]{DM15} (in the case
of Dirac operators), for large enough $|n|$ we have
$$
|\gamma_n| \leq 2 (|\beta_n^- (z_n^*)|+|\beta_n^+ (z_n^*)| ).
$$
Therefore, (\ref{7.21}) implies that for large enough $k$
\begin{equation}
\label{7.23} |\gamma_{n_k}| \leq 2 (|\beta_{n_k}^- (z_{n_k}^*)| +
|\beta_{n_k}^+ (z_{n_k}^*)|) \leq 4 |\beta_{n_k}^+ (z_{n_k}^*)|.
\end{equation}

In view of (\ref{1.137}) in Proposition \ref{bprop1} or (\ref{1.37})
in Proposition \ref{bprop}, for each $z \in [z_n^-,z_n^+]$ and all
$n \in\mathcal{M}$ with large enough $|n|$ we have
\begin{equation}
\label{7.24} |\beta_n^\pm (z) - \beta_n^\pm (z_n^*)| \leq
\sup_{[z_n^-,z_n^+]} \left |\frac{\partial \beta_n^\pm }{\partial z}
(z) \right | \cdot |z-z_n^*| \leq \varepsilon_n |\gamma_n|,
\end{equation}
with $\varepsilon_n  \to 0 $  as $|n|\to \infty.$ Therefore, by
(\ref{7.23}) and (\ref{7.24}) it follows that
\begin{equation}
\label{7.26} |\beta_{n_k}^+ (z)| \geq |\beta_{n_k}^+ (z_{n_k}^*)| -
4 \varepsilon_{n_k} |\beta_{n_k}^+ (z_{n_k}^*) | = (1-4
\varepsilon_{n_k}) |\beta_{n_k}^+  (z_{n_k}^*)|.
\end{equation}

On the other hand, (\ref{7.23}) and (\ref{7.24}) imply that
$$
|\beta_{n_k}^- (z)| \leq |\beta_{n_k}^- (z) - \beta_{n_k}^-
(z_{n_k}^*)| + |\beta_{n_k}^- (z_{n_k}^*)| \leq 4\varepsilon_{n_k}
|\beta_{n_k}^+  (z_{n_k}^*)| + |\beta_{n_k}^- (z_{n_k}^*)|.
$$
Thus, since $\varepsilon_{n_k} \to 0,$ we obtain
$$
\frac{|\beta_{n_k}^-  (z)|}{|\beta_{n_k}^+ (z)|} \leq
\frac{4\varepsilon_{n_k} |\beta_{n_k}^+  (z_{n_k}^*)| +
|\beta_{n_k}^- (z_{n_k}^*)|}{(1-4\varepsilon_{n_k})|\beta_{n_k}^+
(z_{n_k}^*)|} = \frac{4\varepsilon_{n_k} +
t_{n_k}(z_{n_k}^*)}{1-4\varepsilon_{n_k}} \to 0,
$$
i. e., (\ref{7.22}) holds with $\eta_k =\frac{4\varepsilon_{n_k} +
t_{n_k}(z_{n_k}^*)}{1-4\varepsilon_{n_k}}.  $

\end{proof}

Now one may follow  p. 754 in \cite{DM15} (in Russian original p.
170)  in order to complete the proof.

\end{proof}
\bigskip

2. Theorem \ref{crit} provides a general criterion for Riesz basis
property of the system of root functions of Hill operator or Dirac
operator subject to periodic or anti-periodic boundary conditions.
It extends and slightly generalizes \cite[Theorem 1]{DM25} (or
\cite[Theorem 2]{DM25a}) in the case of Hill operators, and
\cite[Theorem 12]{DM26} in the case of Dirac operators.

Theorem \ref{crit} is an effective criterion for analyzing the
existence or non-existence of Riesz bases consisting of root
functions of Hill or Dirac operators. We refer to our papers
\cite{DM25a, DM25, DM26} for concrete applications (see also
\cite[Theorem 71]{DM15}).

Now we give examples of classes of Hill operators with singular
potentials which system of root functions has (or has not) the Riesz
basis property.

\begin{Example}
\label{E11} Let $\mathcal{A} \subset (0,\pi)$ be countable, and let
\begin{equation}
\label{7.44} v(x) = \sum_{k\in \mathbb{Z}}  \sum_{\alpha \in
\mathcal{A}} g(\alpha) \delta (x-\alpha - k \pi) -
\frac{1}{\pi}\sum_{\alpha \in \mathcal{A}} g(\alpha)
\end{equation}
with
\begin{equation}
\label{7.45} \exists \alpha^*: \quad  |g(\alpha^*)| > \sum_{\alpha
\in \mathcal{A}\setminus \{\alpha^*\}} |g(\alpha)|.
\end{equation}
Then the system of root functions of $L_{Per^\pm} (v) $ has the
Riesz basis property.
\end{Example}
(The function $v$ in (\ref{7.44}) lies in $H^{-1}_{per}$ as it
follows from \cite[Theorem 3.1 and Remark 2.3]{HM01} or
\cite[Proposition 1]{DM16}.)

\begin{proof}
Indeed, (\ref{7.44}) implies that the Fourier coefficients of $v$
\begin{equation}
\label{7.47} V(k) = \frac{1}{\pi} \sum_{\alpha \in \mathcal{A}}
g(\alpha) e^{ik\alpha},  \quad k \in 2\mathbb{Z},
\end{equation}
satisfy
\begin{equation}
\label{7.48}  \exists A>0: \quad \frac{1}{A} \leq |V(k)| \leq A,
\quad \forall k \in 2\mathbb{Z}.
\end{equation}
Recall that by (\ref{22.34}) $\beta_n^\pm (v,z) = V(\pm 2n) +
\sum_{k=1}^\infty  S^\pm_k, $ with $S_k $ defined by (\ref{22.35}).
In view of (\ref{22.35}) and (\ref{7.48}),
 $$ |S^\pm_k| \leq  \sum_{j_1, \ldots, j_k \neq \pm n}
\frac{A^{k+1}} {|n^2 -j_1^2 +z| \cdots |n^2 - j_k^2 +z|}.
$$
For  $|z|< n/2, $ we have
$$
|n^2 -j^2 +z| \geq |n^2 -j^2| -n/2  \geq \frac{1}{2}|n^2 -j^2| \quad
\text{for} \;\; j \neq \pm n, \;\; j-n \in 2\mathbb{Z}.
$$
Therefore,
 $$ |S^\pm_k| \leq  \sum_{j_1, \ldots, j_k \neq \pm n}
\frac{(2A)^{k+1}} {|n^2 -j_1^2| \cdots |n^2 - j_k^2|} \leq
(2A)^{k+1} \left (\sum_{j \neq \pm n} \frac{1}{|n^2 -j^2|} \right
)^k.
$$
Now, by the elementary inequality
$$
\sum_{j \neq \pm n} \frac{1}{|n^2 -j^2|} \leq \frac{2\log n}{n},
\quad n \geq 3,
$$
it follows that
$$
|S^\pm_k| \leq (4A)^{k+1}  \left (\frac{\log n}{n} \right )^k.
$$
Thus,  $\sum_{1}^\infty |S^\pm_k| = O ( (\log n)/n),$ so we obtain
\begin{equation}
\label{7.50} \beta_n^\pm (v,z)= V(\pm 2n) +O ((\log n)/n).
\end{equation}
In view of (\ref{7.48}) the latter formula implies (\ref{cr11}),
thus the system of root functions of $L_{Per^\pm} (v) $ has the
Riesz basis property.
\end{proof}
\bigskip

3. Next we use (\ref{7.50}) to explain the claims in Example
\ref{eG}. \vspace{2mm}\\ {\em Proof of Claims (i) and (ii) in
Example \ref{eG}}.

Proof of (i). In view of (\ref{42.13}), the Fourier coefficients
$V(m), \; m \in 2\mathbb{Z},$ of the potential $v$ in Example
\ref{eG} are given by
$$
V(m) = \begin{cases}  0  &   m \leq 0,\\
c(m/2)       &     m >0.
\end{cases}
$$
Since $V(m) =0 $ for $m \leq 0,$  one can easily see from Formulas
(\ref{22.34}) and (\ref{22.35}) that
$$
\beta_n^- (v;z) \equiv 0 \quad  \forall n >N_*, \; |z| \leq n.
$$
On the other hand, by (\ref{42.14}),
$$ \exists A>0: \quad 1/A \leq |V(m)| \leq A \quad \forall m \in 2\mathbb{N},
$$
so the same argument as above proves that (\ref{7.50}) holds. Since,
by (\ref{42.14}), we have  $|V(2n)|>1/A $,  it follows that
\begin{equation}
\label{7.52} \beta_n^+ (v;z) = V(\pm 2n) +O ((\log n)/n) \neq 0
\quad \text{if} \;\; n >N_{*}.
\end{equation}

Fix an $n>N_{*}.$ By Proposition \ref{bprop1}, the equation
(\ref{be1}), that is  $$ (z-\alpha_n (z;v))^2 = \beta^+_n (z;v)
\beta^-_n (z;v) $$ has exactly two (counted with multiplicity) roots
in the disc $|z| < n/4. $ Since $\beta_n^- (v;z) \equiv 0, $ now
this equation has one double root, say $z_n^*, $ and the matrix $$
\left [
\begin{array}{cc} \alpha_n (v,z_n^*)-z_n^*  & \beta^-_n (v;z_n^*)
\\ \beta^+_n (v;z_n^*) &  \alpha_n (v,z) \alpha_n (v,z_n^*)-z_n^*
\end{array}
\right ] =\left [
\begin{array}{cc} 0  & 0
\\ \beta^+_n (v;z_n^*) &  0
\end{array}
\right ]
$$
is Jordan. In view of Lemma \ref{lem1}(c), this implies that all
$E_m $ with $m>N_*$ are Jordan, i.e., (i) in Example \ref{eG} holds.
\vspace{2mm}

Proof of (ii). By the proof of (i)  we have, for large enough $n,$
\begin{equation}
\label{7.54} \gamma_n =0, \quad \beta^-_n (v;z_n^*)=0,  \quad
\frac{1}{2A} \leq |\beta^+_n (v;z_n^*)| \leq 2A.
\end{equation}
Therefore, by \cite[Theorem 37, (7.30)]{DM21} it follows for
$n>N_{**}$ that
\begin{equation}
\label{7.55}
  \frac{1}{144A} \leq  \frac{1}{72} |\beta^+_n (v;z_n^*)|
    \leq    |\mu_n - \lambda_n^+| \leq   58 |\beta^+_n (v;z_n^*)|
    \leq 116 A.
\end{equation}

We set
$$
f_n = u_{2n},  \quad   \xi_n = \|u_{2n-1} \|^{-1}, \quad
     \varphi_n = \xi_n \cdot u_{2n-1}.
$$
Then (\ref{42.10}) takes the form
$$
L \,f_n = \lambda_n^+ f_n, \quad L\, \varphi_n = \lambda_n^+
\varphi_n + \xi_n \cdot f_n,
$$
so now we are using the notations of \cite[Lemma 30]{DM21} (or
\cite[Lemma 59]{DM15}) and can apply the related Fundamental
Inequalities.

By the inequalities
$$
|\mu_n - \lambda_n^+| \leq 4 \xi_n  + 4 |\gamma_n|
$$
$$
\xi_n \leq 4 |\gamma_n| + 2(|\beta^-_n (v;z_n^*)|+|\beta^+_n
(v;z_n^*)|)
$$
(see \cite{DM15}, p. 741;  p. 156 in Russian original) it follows,
in view of (\ref{7.54}) and (\ref{7.55}), that
$$
\xi_n \sim |\mu_n - \lambda_n^+| \sim |\beta^+_n (v;z_n^*)|).
$$
Therefore,
$$
0 <  \inf  \{\xi_n \},      \qquad \sup \{\xi_n \}< \infty,
$$
so the system $\{ u_{2n}, \, u_{2n-1}, \; n>N_{*}\} $ is a Riesz
basis in its closed linear span. This completes the proof of Claim
(ii) in Example \ref{eG}.
\bigskip

\section{Fundamental inequalities and criteria for Riesz basis property}

  1. Now we have to analyze carefully  2D-blocks, $P_m, \, E_m = Ran P_m$ and
pairs of root-functions $\{u_{2m-1}, u_{2m} \}.$

  As a matter of fact it has been done -- just in the form which perfectly fits
to our needs coming from Criterion \ref{crit6} -- in our papers
\cite{KaMi01,DM5,DM15, DM21}. T. Kappeler and B. Mityagin
\cite[Theorem 4.5]{KaMi01}, in the case of Hill operator with
$L^2$-potential proved the inequality
\begin{equation}
\label{5.01} |\mu -\lambda^+| \leq 2 K_{10} (|\xi | +2 |\gamma|)
\end{equation}
(see notations in (\ref{4.1}) - (\ref{4.6}) below). P.Djakov and B.
Mityagin \cite[Lemma 10, Inc. (4.32)]{DM5} succeeded to go to the
opposite direction and proved the inequality
\begin{equation}
\label{5.02} |\xi |\leq  6 |\gamma| + 8|\mu -\lambda^+|
\end{equation}
(Notice that the constants may change because in \cite{KaMi01} and
\cite{DM5} the interval $I =[0,1],$  not $[0,\pi] $ as in the
present paper.)

All these results are presented in \cite{DM15} and the proofs are
written in the way which covers the case of 1D Dirac operator as
well -- see Section 4.2 and 4.3 there. Moreover, these proofs could
be extended to the case of Hill operators with $H^{-1}_{per}$
potentials as soon as we prove (\ref{2.4}) for the deviations $P_n
-P_n^0.$ This is done in \cite[Section 9.2, Proposition 44 and
Theorem 45]{DM21} even in a stronger form
\begin{equation}
\label{5.03} \|P_n -P_n^0\|_{L^1 \to L^\infty} \to 0  \quad
\text{as} \;\; n \to \infty
\end{equation}
-- see \cite[(9.7), (9.8), (9.84)]{DM21}.  Analogues of the
inequalities (\ref{5.01}) and (\ref{5.02}) are inside of the proof
of Lemma 30 there. \bigskip

2. We fix $m$ to consider $E = E_m = Ran P_m, \, dim E = 2,$ with
$m$ large enough. For a while we suppress an index $m$ and write
\begin{equation}
\label{4.1} f=u_{2m}, \quad h=u_{2m-1}, \quad \gamma = \lambda_m^+
-\lambda_m^- \neq 0
\end{equation}
with
\begin{equation}
\label{4.2} L_{bc} f = \lambda^+ f, \quad  L_{bc} h = \lambda^- h,
\quad \|f\|=\|h\|=1
\end{equation}
and such a normalization that
\begin{equation}
\label{4.3} h= af + b \varphi, \quad \langle \varphi, f\rangle =0,
\quad a\geq 0, \;b >0, \;\; a^2 +b^2 =1.
\end{equation}
Notice that
\begin{equation}
\label{4.4} \langle u_{2m}, u_{2m-1}   \rangle =\langle  f, h
\rangle= a, \quad \kappa := (1-a^2)^{-1/2} = 1/b.
\end{equation}
Moreover,
\begin{equation}
\label{4.5} L_{bc} \varphi = (\lambda^+ -\gamma)\varphi + \xi f,
\quad \xi= -\frac{a}{b} \gamma.
\end{equation}
For $\mu = \mu_m$  put
\begin{equation}
\label{4.6} L_{Dir} g = \mu g, \quad  \|g\|=1.
\end{equation}
Then -- see \cite[formula (4.32)]{DM15} and the lines which follow
-- for some $\tau, \; 1/2 \leq |\tau|,$  by \cite[(4.28)]{DM15}
\begin{equation}
\label{4.7} \tau (\mu - \lambda^+)g = b (\xi P_{Dir}f - \gamma
P_{Dir}\varphi).
\end{equation}
Put
\begin{equation}
\label{4.8} r = \frac{|\mu - \lambda^+|}{|\lambda^+ -\lambda^-|},
\quad \text{i.e.,} \quad |\gamma|= \frac{1}{r} |\mu - \lambda^+|;
\end{equation}
then
\begin{equation}
\label{4.9} \mu - \lambda^+ = \frac{1}{\tau} b \left (  \xi \langle
P_{Dir}f, g  \rangle - \gamma \langle P_{Dir}\varphi, g
\rangle\right )
\end{equation}
and with $\| P_{Dir}\| \leq 3/2$ by (\ref{2.4}) we have
\begin{equation}
\label{4.10} |\mu - \lambda^+| \leq 2\left ( \frac{3}{2}|\xi| +
\frac{3}{2} \cdot \frac{1}{r} |\mu - \lambda^+|  \right ) .
\end{equation}
If $r \geq 6$ it follows that
\begin{equation}
\label{4.11} |\mu - \lambda^+| \leq 6 |\xi | = 6 a\,|\gamma|/b \leq
\frac{6}{b} \cdot |\gamma|,
\end{equation}
and
\begin{equation}
\label{4.12} r\leq 6\kappa, \quad  \kappa \in (\ref{4.4}).
\end{equation}
If $r \leq 6$ of course (\ref{4.12}) holds because $\kappa \geq 1$.

These relations (\ref{4.11})--(\ref{4.12}) hold for any $m \in M, $
\begin{equation}
\label{4.13}  \mathcal{M} =\{n: \; \gamma_n = \lambda_n^+
-\lambda_n^- \neq 0, \quad n\geq N_*\}.
\end{equation}
For $\Delta \subset \mathcal{M}$ set
\begin{equation}
\label{4.14} U_\Delta =\{u_{2m-1}, u_{2m} : \; m \in \Delta\}
\end{equation}
and
\begin{equation}
\label{4.15} H_\Delta = \text{the closure of} \; Lin\,Span
\;U_\Delta.
\end{equation}
\begin{Proposition}
\label{prop4.5}
 If the system $U_\Delta$ is a basis in $H_\Delta$ then
\begin{equation}
\label{4.16} \kappa (\Delta) = \sup \{(1-|\langle u_{2m-1}, u_{2m}
\rangle|^2)^{-1/2}: \; m \in \Delta\} < \infty
\end{equation}
is finite, and
\begin{equation}
\label{4.17} R_\Delta = \sup_{m \in \Delta} \frac{|\mu -
\lambda^+|}{|\lambda^+ -\lambda^-|}  \leq 6 \kappa(\Delta) <\infty.
\end{equation}
\end{Proposition}

\begin{proof}
 With proper adjustments of indexation (see the remark in the first
 paragraph of Section 4.2)
Criterion \ref{crit6}, Formula (\ref{3.21}), imply that if
$U_\Delta$ is a basis then (\ref{4.16}) holds. By
(\ref{4.11})--(\ref{4.12}) for each individual $m \in \Delta$
\begin{equation}
\label{4.18}r_m = \frac{|\mu_m - \lambda_m^+|}{|\lambda_m^+
-\lambda_m^-|} \leq 6\kappa_m.
\end{equation}
Taking {\em supremum} over $m \in \Delta$ we get (\ref{4.17}).
\end{proof}
\bigskip

3. Now we want to complement the inequality
(\ref{4.11})--(\ref{4.12}) with estimates of $\kappa = 1/b$ from
above in terms of $r \in (\ref{4.8})$ ($m$ is suppressed).
  It immediately follows from the inequality
\begin{equation}
\label{4.19} |\xi|  \leq 8 |\gamma| + 36 |\mu - \lambda^+|
\end{equation}
-- see lines after formula (4.59) on p. 745 in \cite{DM15} (p. 161
in Russian original). Indeed with $\gamma \neq 0$  (\ref{4.19})
together with (\ref{4.5}) and (\ref{4.8}) imply
\begin{equation}
\label{4.20} |\xi| \leq  \frac{1}{b}|\,\gamma|  \leq (8+36r)|\gamma|
\end{equation}
so
$$b \geq \frac{\sqrt{3}}{2}, \quad  \frac{1}{b} \leq
\frac{2}{\sqrt{3}}< 2,  \quad \text{or} \; \;  b \leq
\frac{\sqrt{3}}{2},  $$ and
\begin{equation}
\label{4.21} \frac{1}{2b} \leq \frac{\sqrt{1-b^2}}{b} \leq 4(2+9r);
\end{equation}
Therefore, in either case
\begin{equation}
\label{4.22} \kappa = \frac{1}{b}  \leq 16 + 72 r.
\end{equation}
With these inequalities Criterion \ref{crit6}, its second part,
implies with notations (\ref{4.16}), (\ref{4.17}) the following.

\begin{Proposition}
\label{prop2.7} If $R_\Delta < \infty $ then
\begin{equation}
\label{4.23} \kappa (\Delta) \leq 16 + 72 R_\Delta
\end{equation}
and the system $U_\Delta$ is a Riesz basis in $H_\Delta.$
\end{Proposition}

\begin{proof}
 Again, individual inequalities
\begin{equation}
\label{4.24} \kappa_m \leq 16 + 72 r_m, \quad m\in \Delta
\end{equation}
hold by (\ref{4.22}). With $R_\Delta$  being finite if we take {\em
supremum} over $m \in \Delta$ in (\ref{4.22}) we get (\ref{4.23}).
Then Criterion \ref{crit6} claims that $U_\Delta$ is a Riesz basis
in $H_\Delta$.
\end{proof}

4. {\em Fundamental inequalities} (\ref{4.11}) and (\ref{4.19}) for
individual $m$ and Propositions \ref{prop4.5} and \ref{prop2.7}
where a subset $\Delta $ could be chosen as we wish emphasize that
neither Dirichlet eigenvalues $\mu_m, \,m \not \in \Delta, $  nor
$Per^+$ or $Per^-$ eigenvalues $\lambda^\pm $ for  $m \not \in
\Delta $ could have any effect on $R_\Delta $  or $\kappa (\Delta).$
In particular, Dirichlet eigenvalues with even (or odd) indices have
no effect {\bf whatsoever} when convergence of spectral
decompositions related to $Per^-$ (or $Per^+$ correspondingly) is
considered.

We can combine Propositions \ref{prop4.5} and \ref{prop2.7} and
claim (for all four cases listed in Section 4.2 in the line prior to
(\ref{42.1})) the following.

\begin{Theorem}
\label{thmM} Let $L_{Per^\pm}(v) $  be either the Hill operator with
$L^2 $ or $H^{-1}_{per}$-potential $v$ or the Dirac operator with
$L^2$-potential $v,$  subject to periodic $Per^+$ or anti-periodic
$Per^-$ boundary conditions. Then the following conditions are
equivalent:\bigskip

(1) The system of root functions of $L_{Per^\pm}(v) $ contains a
Riesz basis in $L^2 ([0,\pi])$ (respectively in $L^2 ([0,\pi])^2.$
\bigskip

(2) The system $\{u_j\}$ defined in (\ref{500}) is a Riesz basis in
$L^2 ([0,\pi])$ (respectively in $(L^2 ([0,\pi]))^2.$
\bigskip

(3) The system $\{u_j\}$ is a basis in $L^2 ([0,\pi])$ (respectively
in $(L^2 ([0,\pi]))^2.$
\bigskip

 (4)  $\kappa
(\mathcal{M}) := \sup \,\{(1-|\langle u_{2m-1}, u_{2m} \rangle|^2
)^{-1/2}: \; m \in \mathcal{M}\} <\infty. $
\bigskip

 (5) $  R (\mathcal{M}) := \sup \left \{\frac{|\mu_m -
\lambda_m^+|}{|\lambda_m^+ -\lambda_m^-|}: \;  m \in \mathcal{M}
\right \} < \infty. $
\bigskip

(6) The system $\{u_j\}$ is a basis in a separable r.i.f.s. $E$ such
that for some $1 < a \leq b < \infty$
$$
L^{a} \supset E \supset L^b, \quad \| g\|_{L^{a}} \leq \| g\|_{E}
\leq \| g\|_{L^{b}} \quad \forall g \in L^\infty.
$$

(7) With $\beta^\pm_n (v,z)$ defined in (\ref{p1}), and $t_n (z)=
|\beta^-_n (v,z)/\beta^+_n (v,z)|$
\begin{equation}
\label{cr1} 0< \liminf_{n\in \mathcal{M}} t_n (z_n^*), \quad
\limsup_{n\in \mathcal{M}} t_n (z_n^*)  < \infty,
\end{equation}
where $z_n^* = \frac{1}{2}(\lambda_n^+ +\lambda_n^-) - n^2$ in the
Hill case and $z_n^* = \frac{1}{2}(\lambda_n^+ +\lambda_n^-) - n$ in
the case of Dirac operators.
\end{Theorem}
(Recall that $\beta_n^\pm (v;z) $ are introduced in Section 2.5,
Lemma \ref{lem1}; see their basic properties in Propositions
\ref{bprop1} and \ref{bprop}).

\begin{proof}
The equivalence of Conditions (1) -- (5) follows from Propositions
\ref{prop4.5} and \ref{prop2.7} and Corollary \ref{cor7}. Conditions
(6) and (7), and their equivalence to (1) -- (5) are explained in
Sections 5, Theorem \ref{thm5.8} and Section 6, Theorem \ref{crit}.
\end{proof}

\end{document}